\documentclass{amsart}
\usepackage[utf8]{inputenc}
\usepackage{amsmath}
\usepackage{amssymb}
\usepackage[T1]{fontenc}
\usepackage{color}
\usepackage{graphicx} % Required for inserting images
\usepackage[backend=biber,style=ieee,%articlein=false,
sorting=nyt]{biblatex}
\usepackage{comment}
\usepackage{tikz}
\usepackage{enumitem}
\usepackage{newunicodechar}
\DeclareUnicodeCharacter{0306}{\u{}}
\theoremstyle{plain}
\newtheorem{theorem}{Theorem}
\newtheorem*{thmc}{Theorem C}

\newtheorem{lemma}{Lemma}

\theoremstyle{remark}
\newtheorem{problem}{Problem}

\theoremstyle{definition}

\renewcommand{\d}{\rho}

\DeclareMathOperator{\interior}{int}
\DeclareMathOperator{\Per}{Per}
\DeclareMathOperator{\Trans}{Trans}
\DeclareMathOperator{\Intrn}{Intrn}

\AtEveryBibitem{%
  \clearfield{issn}%
  \clearfield{isbn}%
  \clearfield{doi}%
  \clearfield{url}%
}

\DeclareBibliographyDriver{book}{%
  \printnames{author}%
  \newunit\newblock
  \printfield{title}%
  \newunit\newblock\setunit*{\addspace}
  \printfield{series}%
  \newunit\newblock\setunit*{\addspace}
  \printfield{volume}%
  \newunit\newblock
  \printlist{publisher}%
  \newunit
  \printlist{location}%
  \newunit
  \printfield{year}%
  \finentry}

\DeclareBibliographyDriver{article}{%
  \printnames{author}%
  \newunit\newblock
  \printfield{title}%
  \newunit\newblock
  \printfield{journaltitle}%
  \newunit\newblock\setunit*{\addspace}
  \iffieldundef{volume}{  }
  {
  \printfield{volume}%
  \newunit\setunit*{\addspace}
  }
  \printfield{year}%
  \newunit\newblock
  \printfield{number}%
  \newunit\newblock
  \printfield{pages}%
  \finentry}

\DeclareBibliographyDriver{misc}{%
  \printnames{author}%
  \newunit\newblock
  \printfield{title}%
  \newunit\newblock
  \printfield{eprint}%
  \newunit\newblock\setunit*{\addspace}
  \printfield{year}%
  \newunit\newblock
  \printfield{url}%
  \finentry}

\DeclareFieldFormat[book]{series}{(#1)}
\DeclareFieldFormat[article,thesis,inbook,incollection,inproceedings]{title}{\textit{#1\isdot}} %article title in italics
\DeclareFieldFormat[article,inbook,incollection,inproceedings]{journaltitle}{#1\isdot} %journal name not in italics
\DeclareFieldFormat{pages}{#1}
\DeclareFieldFormat[article,incollection,inproceedings]{volume}{\textbf{#1}}
\DeclareFieldFormat[article,misc,incollection,inproceedings]{year}{(#1)}

\title{On mixing and dense periodicity on spaces with a free arc}

\author{Dominik Kwietniak, Filip Wierzbowski}
\date{}
\dedicatory{Dedicated to Ľubomir Snoha on the occasion of his 70th birthday.}
\address[DK]{Jagiellonian University in Kraków, Faculty of Mathematics and Computer Science,  ul. Łojasiewicza 6, 30-348 Kraków, Poland}
\email{dominik.kwietniak@uj.edu.pl}
\address[FW]{Jagiellonian University in Kraków, Faculty of Mathematics and Computer Science, ul. Łojasiewicza 6, 30-348 Kraków, Poland \and Jagiellonian University, Doctoral School of Exact and Natural Sciences, ul. Łojasiewicza 11, 30-348 Kraków, Poland   
}
\email{filip.wierzbowski@doctoral.uj.edu.pl}
\begin{document}

\begin{abstract} We study the dynamics of continuous maps on compact metric spaces containing a free interval (an open subset homeomorphic to the interval $(0,1)$). 
We provide a new proof of a result of M.~Dirbák, Ľ.~Snoha, V.~Špitalský 
[Ergodic Theory Dynam. Systems, vol. \textbf{33} (2013), no. 6, pp. 1786--1812] stating that every continuous and transitive, but non-minimal map of a space with a free interval is relatively mixing, non-invertible, has positive topological entropy, and dense periodic points. 
The key simplification comes from short proofs of two facts. The first states that every weakly mixing map of a space with a free interval must be mixing and have positive entropy. The second states that a transitive but not minimal map of a space with a free interval has dense periodic points and is non-invertible.
\end{abstract}

\maketitle

\section{Introduction} 
The topology of compact one-dimensional spaces like the interval $[0,1]$ or the unit circle influences the dynamics of self-maps of these spaces. As a result, there are theorems in one-dimensional dynamics that are not true for general dynamical systems. Important examples are Sharkovsky's Theorem or results describing transitive maps on these spaces, see  \cite{ALM,Ruette}. A natural way of generalising the theory of iterations of interval maps is to extend these results first from interval maps to circle maps, then to self-maps of topological trees and graphs, see \cite{ALM,Blokh-I,Blokh-II,Blokh-III,Ruette}. It is then natural to consider maps of dendrites as the next step, see \cite{HM,KKM,SST,Spitalsky}, but many results known for the interval are false on dendrites, see \cite{BFK,HM,KOT,Spitalsky}. 

Spaces with a free interval (that is, compact metric spaces containing an open subset homeomorphic with the open interval $(0,1)\subseteq\mathbb{R}$) form an alternative class of spaces for extending the theory of one-dimensional dynamical systems. These spaces are only somewhere one-dimensional. 
A paper of M.~Dirbák, Ľ.~Snoha, V.~Špitalský \cite{DSS} and its continuation \cite{DHMSS}, present a systematic study of dynamical properties of self-maps of spaces with a free interval. Earlier results concerning the dynamics of such spaces are contained in \cite{AKLS,HKO-JDEA,Kawamura}, see the discussion in \cite{DSS}, cf. the more recent paper \cite{Mihokova} and T.~Drwięga's PhD thesis \cite{drwiega_phd}. These papers show that many features of one-dimensional dynamics actually depend on the presence of an interval-like component in the phase space rather than its global structure. Furthermore, there are uncountably many pairwise non-homeomorphic spaces with a free interval (for example, this family contains all dendrites whose branch points are not dense, in particular, all dendrites with a closed set of endpoints or all compactifications of a ray $[0,\infty)$; see \cite{Awartani,dlVM}). An easy example of the space that is covered by these results is the Warsaw circle.

In particular, \cite[Theorem C]{DSS} establishes the following dichotomy for transitive maps on a space with a free interval (we state the theorem below, making explicit some conditions that were implicit in the original formulation, and provide definitions of all terms in the forthcoming sections):
\begin{thmc}
If $f\colon X\to X$ is a continuous transitive map on a space with a free interval, then
\begin{enumerate}
    \item if $f$ is  non-minimal, then $f$ is relatively strongly mixing, non-invertible, has positive topological entropy and dense periodic points;
    \item if $f$ is minimal, then $f$ is invertible, it cyclically permutes a disjoint union of $k\ge 1$ circles and $f^k$ restricted to each of these circles is topologically conjugate to an irrational rotation of the unit circle.
\end{enumerate}
\end{thmc}
Note that 
exactly one of the two statements holds. The second point is a direct consequence of \cite[Theorem A]{DSS}.

The main aim of the present paper is to provide a short and almost self-contained proof of \cite[Theorem C(1)]{DSS}, which we restate as Theorem \ref{thm:DSS-C(1)} below.  
We call this proof \emph{almost self-contained} because we still need two results from the general theory of topological dynamical systems. One is a variant of a result by Kinoshita that guarantees the existence of a dense set of non-transitive points in a non-minimal transitive system (see Theorem \ref{thm:intrans}, \cite[Fact 2]{DY}, \cite[Theorem 4.3.1]{KS}, or  \cite{Kinoshita} for the original proof); the second is a result describing the structure of transitive but not totally transitive systems using the notion of a regular periodic decomposition (RPD); see Lemma \ref{lem:terminal_RPD} and \cite{Banks-ETDS}. We note that these results, together with \cite[Theorem 3.2]{HKO-JDEA} are also used in the original proof of \cite[Theorem C(1)]{DSS}. 
Our proof of \cite[Theorem C(1)]{DSS} is based on two observations that are relatively easy to prove (Theorems \ref{thm:wm-mix-ent} and \ref{thm:trans-min=dp}). We show in Theorem \ref{thm:wm-mix-ent} that every weakly mixing map of a space with a free interval must be mixing and have positive entropy. Theorem \ref{thm:wm-mix-ent} generalises \cite[Theorem 3.2]{HKO-JDEA}. Theorem \ref{thm:trans-min=dp} states that a transitive but not minimal map of a space with a free interval has dense periodic points and is noninvertible. The proof of the latter combines ideas from \cite[Theorem 1.1]{AKLS} and \cite{Silverman}.

Finally, note that although Theorem \ref{thm:wm-mix-ent} is a direct corollary of \cite[Theorem C(1)]{DSS}, it was proved independently by the first named author of the present paper around the same time that the results of \cite{DSS} were being written (this is acknowledged in \cite{DSS}, see reference [Kw11] therein). We also note that the original proof of \cite[Theorem C(1)]{DSS} uses \cite[Theorem 3.2]{HKO-JDEA} and that a reference to Theorem \ref{thm:wm-mix-ent} may replace the reference to \cite[Theorem 3.2]{HKO-JDEA}. Therefore, our proof of \cite[Theorem C(1)]{DSS} seems to be more direct than the original one. It also yields \cite[Theorem 3.2]{HKO-JDEA} as an easy corollary, see Theorem \ref{thm:tt-mix}.

\begin{comment}
    
\end{comment}

\section{Terminology and notation}

Let $(X,\d)$ be a compact metric space. A \emph{map} (of $X$) means here a \emph{continuous} map (from $X$ to itself). Given $n\ge 0$, we write $f^n$ for the $n$-th iterate of the map $f$ with the convention that $f^0=\text{id}_X$. The \emph{orbit} of a point $x$ is the set $\mathcal{O}_f(x)=\{f^n(x):n\ge 0\}$. A \emph{transitive point} is a point $x\in X$ whose orbit is dense in $X$. We call a map $f\colon X\to X$ \emph{transitive} if for every two nonempty open sets $U, V\subseteq X$ there is $k\ge 0$ such that $f^k(U) \cap V \neq \emptyset$. We write $\Trans(f)$ for the set of all transitive points and $\Intrn(f)$ for its complement $X\setminus\Trans(f)$. 
In fact, $f$ is transitive if and only if $\Trans(f)$ is a dense $G_\delta$-set,  
see \cite[Theorem 2.2.2]{KS}. Note that if $f$ is transitive, then the orbit $\mathcal{O}_f(x)$ of an intransitive point $x\in\Intrn(f)$ is nowhere dense. Indeed, if $U=\interior \overline{\mathcal{O}_f(x)} \neq \emptyset$, then $U$ would contain a transitive point $y$ but then the orbit of $y$ would never intersect the nonempty open set $X \setminus \overline{\mathcal{O}_f(x)}$, which gives a contradiction. A map $f$ is \emph{totally transitive} if for every $n\ge 1$ the map   $f^n$ is transitive. We say that $f\colon X\to X$ is \emph{weakly mixing} if the product map $f \times f\colon X\times X \to X \times X$ is transitive. By a result of Furstenberg \cite[Proposition II.3]{Fursetberg} weak mixing is equivalent to the fact that for every $m\ge 1$ the $m$-fold product map $f^{\times m}$ on the $m$-fold Cartesian product
\[
X^m=\underbrace{X\times\ldots\times X}_{\text{$m$ times}}
\]
is transitive (see also \cite{Banks-DCDS}).
Finally, $f$ is \emph{mixing} if for every nonempty open sets $U, V\subseteq X$ there exists $N \ge 0$ such that $f^n(U) \cap V \neq \emptyset$ for every $n \geqslant N$. For a map $f\colon X\to X$ we have the following implications: 
\[\text{$f$ mixing} \Longrightarrow \text{$f$ weakly mixing}\Longrightarrow \text{$f$ totally transitive}\Longrightarrow \text{$f$ transitive}.\]
Only the middle implication is a non-trivial one (see e.g. \cite[Theorem 2.7]{Ruette} for the proof). None of the implications above can be reversed.

A set $I \subseteq X$ is an \emph{arc} if $I$ is homeomorphic to the unit interval  $[0,1]\subseteq \mathbb{R}$. In this paper we will follow the terminology used in \cite{DSS}. 

An open subset $J \subseteq X$ is a \emph{free interval} if $J$ is homeomorphic to the open interval $(0,1)\subseteq\mathbb{R}$. If, in addition, $X$ is connected and for each $x \in J$ the set $X \setminus \{x\}$ has exactly two connected components, then we say that $J$ is a \emph{disconnecting interval}. The latter definition comes from \cite{AKLS}. In \cite[Appendix B]{DSS} relations between several ‘natural’ definitions of a 
disconnecting interval are discussed.

An arc $I$ is a \emph{free arc} (in $X$) if there is a homeomorphism $\varphi\colon [0,1] \to I$ such that the set $\varphi((0,1))$ is a free interval, that is, it is open in $X$. In fact, if $I$ is a free arc in $X$, then for every homeomorphism $\psi\colon [0,1] \to I$ we have that the set $\psi((0,1))$ is open in $X$. (Let $\varphi$ be a homeomorphism from the definition of a free arc (closed interval). Then
$\Psi=\psi^{-1}\circ \varphi\colon [0,1]\to [0,1]$ is a homeomorphism, so $\Psi((0,1))=(0,1)$, hence $\varphi((0,1))=\psi((0,1))$.)

Note that our terminology differs from \cite{HKO-JDEA}: what we call a free interval is called a free arc in \cite{HKO-JDEA}, while our free arc corresponds to their closed interval.

The closure of a free interval in $X$ need not be a free arc. For example, there is a free interval which is an open dense subset of the topologist sine curve ($\sin\frac1x$ continuum). But a free interval always contains a free arc; therefore, a space $X$ has a free arc if and only if it has a free interval. 

Given two free arcs $I, J \subseteq X$ and a map $f\colon X\to X$ we say that $I$ \emph{$f$-covers} $J$ if there exists a free arc $K \subseteq I$ such that $f(K) = J$. Below, we recall basic properties of the $f$-covering relation. 
These results are based on the following elementary observation.

\begin{lemma}\label{contSurj}
            If $g\colon [0,1] \to [0,1]$ is a continuous surjection, then for any choice of $0 \leqslant c < d \leqslant 1$ there exist $0 \leqslant a < b \leqslant 1$ such that $g([a,b]) = [c,d]$.
\end{lemma}        
       
   \begin{proof}
            Since $g$ is a surjection, we find $x_0, x_1 \in [0,1]$ such that $g(x_0) = c$ and $g(x_1) = d$. Without loss of generality $x_0 < x_1$. Take
            \[y_0 := \sup \{x \in [x_0, x_1]: g(x) = c\}.\]
            The point $y_0$ is well-defined as the supremum of a nonempty set bounded from above.
            Because $g$ is continuous, we have $g(y_0)=c$, so $y_0< x_1$. Now take
            \[y_1 := \inf \{x \in [y_0, x_1]: g(x) = d\}.\]
            By continuity of $g$, we have $g(y_1)=d$, so $y_0 < y_1$. Furthermore, $[c,d] = [g(y_0), g(y_1)] \subseteq g([y_0,y_1])$ by the intermediate value theorem. 

            If there existed $x$ such that $y_0<x<y_1$ and $g(x) < c$, then using again the intermediate value property we would obtain $y \in (x,y_1)$ such that $g(y) = c$ and $y_0<y$, which contradicts the choice of $y_0$. By a similar argument there is no  $x$ satisfying $y_0<x<y_1$ and $g(x) > d$. Hence $g([y_0,y_1]) = [c,d]$.
        \end{proof}
    
The next result, Lemma \ref{f-cover} is adapted from \cite[p. 590]{AdRR}. The result is well-known for graph maps, but since we will apply it in a more general setting, we provide a proof.
\begin{lemma}\label{f-cover}
    Let $X$ be a compact metric space, $I, J, K, L \subseteq X$ be free arcs in $X$, and let $f, g\colon X \to X$ be maps. Then
    \begin{enumerate}[label = (\roman*)]
        \item \label{f-cover-i} If $I$ $f$-covers $I$, then there exists $x \in I$ such that $f(x) = x$;
        \item \label{f-cover-ii} If $I \subseteq K$, $L \subseteq J$ and $I$ $f$-covers $J$, then $K$ $f$-covers $L$;
        \item \label{f-cover-iii} If $I$ $f$-covers $J$ and $J$ $g$-covers $K$, then $I$ $(g\circ f)$-covers $K$;
        \item \label{f-cover-iv} If $\psi\colon [0,1]\to J$ is a 
        homeomorphism and $0\le a<b\le c<d\le 1$ are such that for some
        $0\le t_0\le a$, $b\le t_1\le c$, and  $d\le t_2\le 1$ we have $\psi(t_i)\in f(I)$ for $i=0,1,2$, then $I$ $f$-covers $\psi([t_0,t_1])$ or $\psi([t_1,t_2])$ (possibly both);
        
        \item \label{f-cover-v} If $J \subseteq f(I)$ and $K_1, K_2 \subseteq J$ are two free arcs such that $K_1 \cap K_2$ is at most one point, then $I$ $f$-covers $K_1$ or $I$ $f$-covers $K_2$.
    \end{enumerate}
\end{lemma}

\begin{proof}
We begin with a proof of \ref{f-cover-i}. Let $\varphi\colon [0,1] \to I$ be a homeomorphism such that $\varphi((0,1))$ is open in $X$. Since $I$ $f$-covers $I$, there exists $J \subseteq I$ such that $f(J) = I$. Take $a, b \in [0,1]$ such that $J = \varphi([a,b])$. Then $g := \varphi^{-1}\circ f\circ \varphi\colon [a,b] \to [0,1]$ is a continuous surjection. 
        If $g(x) < x$ for all $x \in [a,b]$, then $1 \notin g([a,b]) = [0,1]$, which is a contradiction; similar reasoning shows that it is not possible that $g(x) > x$ for all $x\in[a,b]$. Thus there are $x_1, x_2 \in [a,b]$ such that $g(x_1) - x_1 \ge 0$ and $g(x_2) - x_2 \le 0$. From the intermediate value property, there exists $x_0 \in [x_1,x_2] \subseteq [a,b]$ such that $g(x_0) = x_0$. Then $f(\varphi(x_0)) = \varphi(x_0)$.

Next, we prove \ref{f-cover-ii}. Let $\varphi\colon[0,1]\to K$, 
$\psi\colon [0,1]\to J$ be homeomorphisms. Let $I'\subseteq I\subseteq K$ be a free arc such that $f(I')=J$. Take $0\le a<b\le 1$ and $0\le c<d\le 1$ such that $\varphi([a,b])=I'$ and $\psi([c,d])=L$. Consider $g\colon [a,b]\to [0,1]$ given by
$g=\psi^{-1}\circ f \circ\varphi|_{[a,b]}$. It follows from our assumption that $g$ is a continuous surjection. Applying Lemma \ref{contSurj} we get the desired conclusion. 

Note that \ref{f-cover-iii} is an immediate consequence of \ref{f-cover-ii}. 
        
For a proof of \ref{f-cover-iv} take a homeomorphism $\varphi\colon [0,1] \to I$. By assumption for each $i\in\{0,1,2\}$ there exists $s_i \in [0,1]$ such that $\psi(t_i) = f(\varphi(s_i))$. We define 
        $$C = \bigcup\{[s_0',s_2'] \subseteq [0,1]: s_1 \in [s_0',s_2'],\ f(\varphi([s_0',s_2'])) \subseteq J\}.$$
Recall that $U=\interior J$ is a neighbourhood of $\psi(t)$ for every $0<t<1$. In particular, $U$ is a neighbourhood of $\psi(t_1)$. Therefore, by continuity of $f\circ\varphi$, the set $(f\circ\varphi)^{-1}(U)$ is an open neighbourhood of every $s\in[0,1]$ such that $f\circ\varphi(s)\in U$. It follows that the set $C$ is nonempty, closed and has nonempty interior.   
        Notice also that $C$ is connected as a union of connected sets with nonempty intersection. Therefore, $C = [s_0',s_2']$ for some $s_0' < s_2'$ and $s_0' \leqslant s_1 \leqslant s_2'$. By Lemma \ref{contSurj} it is now enough to show that there is $i \in \{0,1\}$ such that $[t_i,t_{i+1}] \subseteq g(C)$ where $g := \psi^{-1}\circ f\circ \varphi|_C$. If $C=[0,1]$, then we are done. Otherwise, if $s_0'>0$, then $f(\varphi(s_0'))\in \{\psi(0),\psi(1)\}$, by maximality of $C$. Thus $\varphi([t_i,t_{i+1}]) \subseteq f(\varphi(C))$ for some $i \in \{0,1\}$. An analogous reasoning applies when $s_1'<1$.         
        Using Lemma \ref{contSurj} we obtain that $I$ $f$-covers $\psi([t_i,t_{i+1}])$ for some $i \in \{0,1\}$.

        To finish the proof, note that \ref{f-cover-v} follows from \ref{f-cover-iv}. 
\end{proof}

We refrain from recalling the definition of topological entropy, we refer to \cite{ALM}. We will just need the following easy criterion for positivity of topological entropy, see \cite{ALM, Ruette}.
\begin{lemma}[Horseshoe lemma] If  $K_1, K_2 \subseteq J$ are two disjoint free arcs 
such that $K_1\cup K_2\subseteq f(K_1)\cap f(K_2)$, 
    then $f$ is non-invertible and the topological entropy $h(f)$ of $f$ is positive.
\end{lemma}

\section{When does weakly mixing imply mixing?}
We recall \cite[Lemma 3.1]{HKO-JDEA} and provide a proof for completeness.
\begin{lemma}\label{lem:trans-mix}
    If $f\colon X \to X$ is a transitive map and $x \in X$ is a transitive point such that for each neighbourhood $W$ of $x$ we can find an $N$ with $f^n(W) \cap W \neq \emptyset$ for all $n \geqslant N$, then $f$ is mixing.
\end{lemma}

\begin{proof} Fix nonempty open sets $U, V\subseteq X$.
    Let $l\ge 0$ be such that $W'=U\cap f^{-l}(V)\neq\emptyset$. Since $W'$ is open and $x$ is a transitive point, there is $k\ge 0$ such that
    \[W = f^{-k}(W') =f^{-k}(U \cap f^{-l}(V))\] 
    is a neighbourhood of $x$. Take $N \ge 0$ such that $f^{n}(W) \cap W \neq \emptyset$ for $n \geqslant N$. Hence, for  $n\ge N$ and $y\in W\cap f^{-n}(W)$ we have 
    $f^k(y),f^{n+k}(y)\in W'$, in particular $f^k(y)\in U$ and $f^{n+k+l}(y)\in V$. %
    Therefore, $f^m(U) \cap V\neq\emptyset$ for $m \geqslant N+l$. 
\end{proof}

\begin{theorem}\label{thm:wm-mix-ent}
    If $X$ has a free interval, then every weakly mixing map of $X$ is mixing, non-invertible, and has positive topological entropy.
\end{theorem}

\begin{proof} 
    Let $J$ be a free interval in $X$ and let $\varphi\colon (0,9)\to J$ be a homeomorphism.   Define   
    \begin{gather*}
        A = \varphi([1,2]),\quad B = \varphi([3,4]),\quad C = \varphi([5,6]),\quad D = \varphi([7,8]),\\ I_j = \varphi((2j,2j+1))\ \text{for}\ 0 \leqslant j \leqslant 4,\\
I^* = I_0\times I_1\times \ldots \times I_4\subseteq X^5.
    \end{gather*}
Since $g = f^{\times 20}\colon X^{20} \to X^{20}$ is transitive (recall that, $f^{\times 20}$ stands for $20$-fold Cartesian product of $f$) and $A,B,C,D$ have nonempty interiors, we can find $k \geqslant 0$ such that
\[g^k(A^5 \times B^5 \times C^5\times D^5) \cap I^* \times I^* \times I^* \times I^*\neq \emptyset.\]    
We conclude from Lemma \ref{f-cover}\ref{f-cover-iv} that each of the arcs $A, B,C, D$ has to $f^k$-cover at least three distinct arcs among $A, B, C, D$. In particular, there are $K_0,K_1 \in \{A, B, C,D\}$ with $K_0\cap K_1=\emptyset$ such that $A$, respectively $D$,  $f^k$-covers $K_0$ and $K_1$. 

Now, note that $K_i$ (for $i=0,1$) also has to $f^k$-cover at least one arc in $\{A,D\}$. That is, for $i=0,1$ there is $E_i\in\{A,D\}$ such that $K_i$ $f^k$-covers $E_i$ and $E_i$ $f^k$-covers every interval $L\in\{K_0,K_1\}$. 
It follows from Lemma \ref{f-cover}\ref{f-cover-iii} that for $i=0,1$ the interval $K_i$ $f^{2k}$-covers every interval $L\in\{K_0,K_1\}$, in particular we have 
    \[
    K_0\cup K_1\subseteq f^{2k}(K_0)\cap f^{2k}(K_1).
    \]
Hence $f^{2k}$ has a two horseshoe, so it is non-invertible and the entropy is positive. 

To prove mixing, we need to modify this reasoning. We use transitivity of $h\times h$, where $h = f^{\times 5}\colon X^5 \to X^5$ to find $n \geqslant 0$ such that
    \[(h\times h)^n(K_0^5\times K_0^5) \cap \left(I^* \times h^{-1}(I^*)\right)\neq \emptyset.\]
    In other words, $f^j(K_0) \cap I_i \neq \emptyset$ for $j = n, n+1$ and $i = 0,1, \ldots, 4$. It follows from Lemma \ref{f-cover}\ref{f-cover-iv} that there are arcs $E, E' \in \{A,D\}$ such that $K_0$ $f^n$-covers $E$ and $f^{n+1}$-covers $E'$. But, as we have seen above, $K_0$ is $f^k$-covered by both $A$ and $D$. Therefore by Lemma \ref{f-cover}\ref{f-cover-iii} $K_0$ $f^{n+k}$-covers and $f^{n+k+1}$-covers itself. Applying Lemma \ref{f-cover}\ref{f-cover-iii} again, we see that $K_0$ $f^m$-covers $K_0$ for every integer $m$ in the cofinite set 
    \[\{\alpha\cdot (n+k) + \beta\cdot (n+k+1): \alpha, \beta \in \mathbb{N}\}.\]
    Now let $x\in \Trans(f)$ belong to a free arc. 
    Fix an open neighbourhood $W$ of $x$. Without loss of generality, we assume  that $W=J$ is a free interval in the above reasoning. Since we know from the above that the set
    $\{n\ge 0: f^n(W)\cap W\neq \emptyset\}$  is cofinite, we conclude that $f$ is mixing by Lemma \ref{lem:trans-mix}.
\end{proof}

We can now provide a simple proof of \cite[Theorem 3.2]{HKO-JDEA}. 
\begin{theorem}\label{thm:tt-mix}
    If $X$ is a compact and connected metric space with a disconnecting interval, then each totally transitive map of $X$ is mixing.
\end{theorem}

\begin{proof}
    By \cite[Theorem 1.1]{AKLS} each transitive map of $X$ has a dense set of periodic points. Every totally transitive map with a dense set of periodic points is weakly mixing, see \cite[Theorem 1.1]{Banks-ETDS}.  The result now follows from Theorem \ref{thm:wm-mix-ent}.
\end{proof}

\section{Transitivity and dense periodicity}
A point $x\in X$ is \emph{periodic} if there is $n>0$ such that $f^n(x)=x$. Let $\Per(f)$ denote the set of periodic points of $f$. 
We will need the following lemma stating that if periodic points for a transitive map are somewhere dense, then the set of periodic points is dense in $X$. The result is probably well-known, but we were unable to find a proof other than the proof of \cite[Lemma 8.5]{drwiega_phd}, so we provide our own.
\begin{lemma}\label{lem:reduction}
    If $f\colon X \to X$ is a transitive map and for a nonempty open set $W \subseteq X$ we have $W \subseteq \overline{W \cap \Per(f)}$, then $\overline{\Per(f)} = X$.
\end{lemma}
\begin{proof}
Fix a nonempty open set $V\subseteq X$. By transitivity, there is $k\ge 0$ such that $U=W\cap f^{-k}(V)\neq\emptyset$. Using our assumption, we find $x\in\Per(f)\cap U$. Hence $f^k(x)\in \Per(f)\cap V$, which ends the proof.    
\end{proof}
A map $f\colon X \to X$ is \emph{minimal} if $\Trans(f)=X$. It is known that a transitive map is either minimal, or the set $\Intrn(f)$ of points whose orbits are non-dense is dense in $X$. This is a result due to Kinoshita \cite{Kinoshita}, see \cite[Theorem 4.3.1]{KS} or \cite[Fact 2]{DY} for other proofs.
\begin{theorem}\label{thm:intrans}
    If $X$ has no isolated points and $f\colon X\to X$ is a map, then $\Intrn(f)$ is either empty or dense in $X$.
\end{theorem}

The main purpose of this section is to provide a direct proof of the following theorem:

\begin{theorem}\label{thm:trans-min=dp}
    %Let $X$ be a compact metric space with an open interval and 
    If $f\colon X \to X$ is a transitive but not minimal map and $X$ has a free interval, then $\Per(f)$ is dense in $X$.
\end{theorem}

\begin{proof}
    Let $J$ denote the free interval in $X$ and $\varphi\colon (0,1) \to J$ be a homeomorphism. 
    Since $f$ is transitive and $X$ is infinite, $X$ has no isolated points (it is easy to see directly; cf. the discussion after \cite[Theorem 2.2.1]{KS}). As $f$ is non minimal; by Theorem \ref{thm:intrans}, there exist %different intransitive points 
    $p, q \in J\cap\Intrn(f)$ with $p\neq q$.  

    Define $S := \overline{\mathcal{O}_f(p)} \cup \overline{\mathcal{O}_f(q)}$. Since $p, q \in J\cap\Intrn(f)$ we have
    $\overline{\mathcal{O}_f(p)} \neq X$ and $\overline{\mathcal{O}_f(q)}\neq X$. As the closure of an orbit is always an invariant set and $f$ is transitive, the sets $\overline{\mathcal{O}_f(p)}$ and $\overline{\mathcal{O}_f(q)}$ must be nowhere dense. Therefore, $S$  
    is a closed $f$-invariant set, which must be nowhere dense by the Baire Category Theorem. We consider the compact quotient space $Y = X/S$ with the induced map $\tilde{f}\colon Y \to Y$. Note that it is enough to show that  $\Per(\tilde{f})$ is dense in $Y$. By a slight abuse of notation, we still write $f$ to mean $\tilde{f}$. We choose an open interval $I$ which is a connected component of $J\setminus S$. To simplify notation, we identify $I$ with $(0,1)$ using $\varphi$. We also write $\preceq$ for the orientation induced via $\varphi$ on $I$ by the usual order $\le$ on $(0,1)$. As usual, $x\prec y$ means for $x,y\in I$ that $x\preceq y$ and $x\neq y$. By Lemma \ref{lem:reduction} it is enough to show that for any $0 < a < b < 1$ there is a periodic point of $f$ in $(a,b)$. 
    
    Fix some $(a,b) \subseteq I$. Recall that in $Y$ the set $S$ is identified with a single point, which is fixed by $f$. Let
    $$S' := \bigcup_{k=1}^{\infty} f^{-k}(S).$$
    There are two disjoint cases to consider depending on whether $(a,b)$ has empty intersection with $S'$. 
    
    \textit{Case I:} $(a,b)\cap S'=\emptyset$.

    The proof in this case is very similar to the proof of \cite{AKLS} (see also \cite[Lemma 2.14]{Ruette}). Choose $x, y \in \Trans(f)$ with $a\prec x \prec y\prec b$. Find $n,m > 0$  such that $a \prec f^n(x) \prec x$ and $y \prec f^m(y) \prec b$. 
    We claim that $(a,b)\cap S'=\emptyset$ implies $a \prec f^{nj}(x),f^{mj}(y) \prec b$ for every $j\ge 1$. To prove our claim, assume first $f^{nj}(x)\notin (a,b)$ for some $j\ge 2$. Take $M=\min\{j\ge 2: f^{nj}(x)\notin (a,b)\}$. Let $\langle f^{(M-2)n}(x),f^{(M-1)n}(x)\rangle$ stand for the smallest subinterval of $(a,b)$ containing $f^{(M-2)n}(x)$ and $f^{(M-1)n}(x)$. We see that $S \in f^n(\langle f^{(M-2)n}(x),f^{(M-1)n}(x)\rangle)$ contradicting $(a,b)\cap S'=\emptyset$. Similarly, we prove $a \prec f^{mj}(y) \prec b$ for $j\ge 0$. 
    
    If there exists $k > 0$ such that $a\prec f^{kn}(x) \prec x\prec f^{(k+1)n}(x)$, then 
    $$f^{kn}\left([f^n(x),x]\right) \supseteq [f^n(x),x],$$
    so $f^{kn}$ has a fixed point in $[f^n(x),x] \subseteq (a,b)$. An analogous argument works if there exists $l > 0$ such that $f^{lm}(y) \prec y$.

    Suppose that $a\prec f^{kn}(x) \prec x$ and $y \prec f^{lm}(y)\prec b$ for all $k, l > 0$. Then, choosing $k = m$ and $l = n$, we get
    $$f^{mn}\left([x,y])\right) \supseteq [x,y].$$
    Again, this means that $f^{mn}$ has a fixed point in $[x,y] \subseteq (a,b)$. This finishes the argument in Case I.

    \textit{Case II:} $(a,b)\cap S'\neq\emptyset$. The proof in this case 
    is similar to the proof of \cite[Lemma 6.1]{Silverman}. Let $q \in (a,b)$ and $k \geqslant 1$ be such that $f^k(q) = S$. Notice that then $f^j(q) = S$ for $j \geqslant k$. Choose 
    $x,y \in (a,b)\cap \Trans(f)$ with $x\neq y$. Take $m, n > k$ such that $x \prec f^n(x) \prec q$ and $q \prec f^m(y) \prec y$. Let $\Gamma_n \subseteq (a,b) \times X\subseteq X^2$, respectively $\Gamma_m \subseteq (a,b) \times X\subseteq X^2$ denote the graph of $f^n|_{(a,b)}$, respectively of $f^m|_{(a,b)}$. Let $\Delta \subseteq X \times X$ denote the diagonal, that is $\Delta=
\{(x,x)\in X^2:x\in X\}$. If $\Gamma_n \cap \Delta \neq \emptyset$ or $\Gamma_m \cap \Delta \neq \emptyset$, then there is a periodic point in $(a,b)$. Therefore, from now on, we assume that $\Gamma_n \cap \Delta = \emptyset$ and $\Gamma_m \cap \Delta = \emptyset$.

    Since 
    $\Gamma_n \cap \Delta = \emptyset$, 
    the intermediate value property allows us to define
    $$s := \max\{z \in (x,q): f^n(z) = y\}.$$
    By the same argument applied to $\Gamma_m$, we define
    $$t := \min\{z \in (q,y): f^m(z) = x\}.$$
    Finally, let
    $$u := \max\{z \in (x,s): f^n(z) = t\}.$$
    By definition, $u \prec s$.
    \begin{figure}
    \centering
    \begin{tikzpicture}
        %outline
        \draw[thick,-] (0,0) -- (6,0) -- (6,6) -- (0,6) -- (0,0);
        \draw[dashed,gray,-] (0,0) -- (6,6);
        \draw[dashed,gray,-] (3,0) -- (3,3) -- (0,3);
        \draw[dashed,gray,-] (0.8,0) -- (0.8,1.5) -- (0,1.5);
        \draw[dashed,gray,-] (5.2,4.6) -- (0,4.6);
        \draw[dashed,gray,-] (5.2,0) -- (5.2,5.2) -- (0,5.2);
        \draw[dashed,gray,-] (0,0.8) -- (3.52,0.8);
        \draw[dashed,gray,-] (3.52,0) -- (3.52,3.52) -- (0,3.52);
        \draw[dashed,gray,-] (2.1,5.2) -- (2.1,0);
        \draw[dashed,gray,-] (1.6,3.52) -- (1.6,0);
        \draw (0,0) node[anchor=north east] {$S$};
        \draw (6,0) node[anchor=north] {$S$};
        \draw (0,6) node[anchor=east] {$S$};
        %y axis
        \draw (2pt,0.5) -- (-2pt,0.5) node[anchor=east] {$a$};
        \draw (2pt,5.5) -- (-2pt,5.5) node[anchor=east] {$b$};
        \draw (2pt,3) -- (-2pt,3) node[anchor=east] {$q$};
        \draw (2pt,1.5) -- (-2pt,1.5) node[anchor=east] {$f^n(x)$};
        \draw (2pt,4.6) -- (-2pt,4.6) node[anchor=east] {$f^m(y)$};
        \draw (2pt,0.8) -- (-2pt,0.8) node[anchor=east] {$x$};
        \draw (2pt,5.2) -- (-2pt,5.2) node[anchor=east] {$y$};
        \draw (2pt,3.52) -- (-2pt,3.52) node[anchor=east] {$t$};
        %x axis
        \draw (0.5,2pt) -- (0.5,-2pt) node[anchor=north] {$a$};
        \draw (5.5,2pt) -- (5.5,-2pt) node[anchor=north] {$b$};
        \draw (3,2pt) -- (3,-2pt) node[anchor=north] {$q$};
        \draw (0.8,2pt) -- (0.8,-2pt) node[anchor=north] {$x$};
        \draw (5.2,2pt) -- (5.2,-2pt) node[anchor=north] {$y$};
        \draw (3.52,2pt) -- (3.52,-2pt) node[anchor=north] {$t$};
        \draw (2.1,2pt) -- (2.1,-2pt) node[anchor=north] {$s$};
        \draw (1.6,2pt) -- (1.6,-2pt) node[anchor=north] {$u$};
        %f^n
        \fill (0.8,1.5) circle[radius=1.5pt];
        \fill (2.1,5.2) circle[radius=1.5pt];
        \draw (0.8,1.5) .. controls (1.5,1) and (1.5,6) .. (3,6);
        \draw (1.1,4.3) node[anchor=north west] {$f^n$};
        \fill (1.6,3.52) circle[radius=1.5pt];
        %f^m
        \fill (5.2,4.6) circle[radius=1.5pt];
        \fill (3.52,0.8) circle[radius=1.5pt];
        \draw (5.2,4.6) .. controls (4.5,4) and (4.5,2) .. (3,0);
        \draw (4,2.3) node[anchor=north west] {$f^m$};
    \end{tikzpicture}
    \caption{Example illustration of the graphs $\Gamma_n$ and $\Gamma_m$, and points $s$, $t$, $u$ used in the argument.} \label{fig1}
    \end{figure}
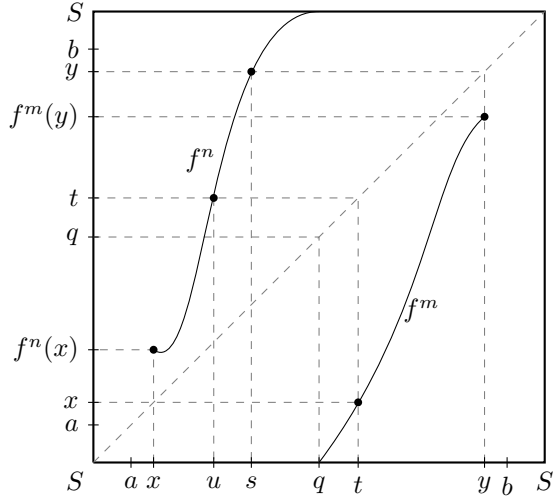

    We have the following inclusions
    $$f^n([u,s]) \supseteq [t,y],$$
    $$f^m([t,y]) \supseteq [x,q] \supseteq [u,s].$$
    This gives $f^{n+m}([u,s]) \supseteq [u,s]$ and so $f$ has a periodic point in $[u,s] \subseteq (a,b)$ by Lemma \ref{f-cover}\ref{f-cover-i}. 
\end{proof}

A set $D \subset X$ is \emph{regular closed} if it is equal to the closure of its interior. A \emph{regular periodic decomposition} (RPD) for a map $f\colon X\to X$ is a finite sequence $\mathcal{D} = (D_0,\ldots,D_{k-1})$ of nonempty regular closed sets that cover $X$ and satisfy $f(D_i) \subset D_{i+1\ \text{mod}\ k}$ for all $0 \leqslant i < k$ and $D_i \cap D_j$ is nowhere dense for $i \neq j$. We say that $k$ is the \emph{length} of a RPD $\mathcal{D} = (D_0,\ldots,D_{k-1})$. We call a RPD \emph{terminal} if it is of maximal length among all RPDs for $f$. By \cite[Theorem 3.1]{Banks-ETDS} a RPD is terminal if and only if $f^k|_{D_i}$ is totally transitive for $0 \leqslant i < k$. We recall the following criterion for the existence of terminal RPD when the space has a free interval, see \cite[Lemma 7]{DSS}. 

\begin{lemma}\label{lem:terminal_RPD}
If $X$ has a free interval $J$ and $f\colon X\to X$ is transitive map with $\Per(f)\cap J\neq\emptyset$, then $f$ has a terminal RPD.
\end{lemma}

We say that $f$ is \emph{relatively mixing} if there is a RPD $\mathcal{D} = (D_0,\ldots,D_{k-1})$ such that $f^k|_{D_i}$ is mixing for each $0 \leqslant i < k$. We are ready to provide a new proof of \cite[Theorem C(1)]{DSS}. Recall that \cite[Theorem C]{DSS} describes a dichotomy for transitive maps on spaces containing a free interval; the first part, referred to as \cite[Theorem C(1)]{DSS} (presented below), describes non-minimal maps, while the second alternative, \cite[Theorem C(2)]{DSS}, describes minimal maps of such spaces.
\begin{theorem}\label{thm:DSS-C(1)}
If $X$ has a free interval and $f\colon X \to X$ is a transitive non-minimal map, then $f$ is relatively mixing, non-invertible, has positive topological entropy, and dense periodic points.
\end{theorem}

\begin{proof} By Theorem \ref{thm:trans-min=dp} we have $\Per(f)$ is dense in $X$.
    By Lemma \ref{lem:terminal_RPD} there exists a terminal RPD for $f$, say $\mathcal{D} = (D_0,\ldots,D_{k-1})$. Thus $f^k|_{D_i}$ is totally transitive for each $0 \leqslant i < k$. Then 
     $f^k|_{D_i}$ for each $0\leqslant i <k$ is a totally transitive map that has dense periodic points, so by \cite[Theorem 1.1]{Banks-ETDS} $f^k|_{D_i}$ is weakly mixing. Finally, at least one $D_i$ contains a free interval, let it be $D_0$. Applying Theorem \ref{thm:wm-mix-ent} to $f^k|_{D_0}$, we see that $f^k|_{D_i}$ is mixing and has positive topological entropy.
\end{proof}
 
It is natural to ask if the results presented here can be generalised even further.

\begin{problem}
For which continua do the appropriate analogs of Theorems \ref{thm:wm-mix-ent} and \ref{thm:tt-mix} hold?    
\end{problem}  

Theorem \ref{thm:wm-mix-ent} is false for general dendrite maps. It follows from the existence of a zero entropy, weakly mixing, non-mixing, and proximal dendrite map. Recall that proximality of a map is equivalent to the fact that the only periodic point of that map is a fixed point which is also the only point that is a minimal point for the system. Such a map of the Ważewski dendrite was constructed in \cite{HM} and proved to have zero entropy in \cite{BFK}. 

Benjamin Vejnar asked whether the assumption of the existence of a disconnecting interval in $X$ in Theorem \ref{thm:tt-mix} could be relaxed to: $X$ is a compact and connected space with uncountably many cut-points (i.e. points whose removal disconnects $X$). The answer is negative already for dendrites. It is known that every point in the dendrite which is not an endpoint is a cut-point. Therefore, the examples from \cite{HM} discussed above show that assuming the existence of uncountably many cut-points cannot replace the existence of a free arc in $X$ in Theorem \ref{thm:tt-mix}.

Nevertheless, it is possible that every weakly mixing map with a dense set of periodic points on a space with uncountably many cut points is mixing (that is, \cite[Theorem 3.2]{HKO-JDEA} generalises to spaces with uncountably many cut points).

Another intriguing question concerns the specification property introduced by R.~Bowen \cite{Bowen}, see also \cite{KLO}. Specification is a kind of stronger, uniform variant of mixing. A.~M.~Blokh proved that an interval map has the specification property if and only if it is mixing and later extended this result to maps of topological graphs, see \cite{Blokh-I,Blokh-II,Blokh-III}. For other proofs, see \cite{AdRR,Buzzi,HKO-ETDS}.

Therefore, a natural problem is how far this result can be generalised.

\begin{problem}
For which continua does every mixing map have the specification property?  Does every mixing map of a space with a free arc have the specification property?  
\end{problem} 

C.~Mouron (personal communication) informed us that the answer is negative for dendrites, but to the best of our knowledge, the problem seems to be open even for the simplest continua with a free interval, like the Warsaw circle.

In \cite{Weiss} there is an example of a Cantor set homeomorphism with a dense set of periodic points and zero entropy, whose only ergodic invariant measures are concentrated on periodic orbits (see \cite{FKL} for more details). This motivates our last problem. We conjecture that the answer to the first half of the question is positive.

\begin{problem}
Does there exist a mixing dendrite map with a dense set of periodic points and zero entropy? Can the only ergodic invariant measures of such a map be the ones concentrated on periodic orbits?   
\end{problem} 

\section*{Acknowledgements}

We would like to thank Ľubom\'{\i}r Snoha for 
a discussion, encouragement, and insightful remarks about the article. 
The publication has been supported by the Flagship Project CENMATRE under the Strategic Programme Excellence Initiative at Jagiellonian University.
We thank the two anonymous referees for their careful reviews and numerous helpful comments and corrections, which helped improve the manuscript. We are particularly grateful for the discussion concerning the role of isolated points and dense orbits in transitive systems, which helped us better understand these issues.

\printbibliography

\begin{thebibliography}{34}

\bibitem{AKLS99}
L.~Alsed\`a, S.~Kolyada, J.~Llibre, and \v{L}.~Snoha,
\emph{Entropy and periodic points for transitive maps},
Trans. Amer. Math. Soc. \textbf{351} (1999), no.~4, 1551--1573.

\bibitem{ARR03}
L.~Alsed\`a, M.~A.~del R\'{\i}o, and J.~A.~Rodr\'{\i}guez,
\emph{Transitivity and dense periodicity for graph maps},
J. Difference Equ. Appl. \textbf{9} (2003), no.~6, 577--598.

\bibitem{ALM00}
L.~Alsed\`a, J.~Llibre, and M.~Misiurewicz,
\emph{Combinatorial dynamics and entropy in dimension one},
Advanced Series in Nonlinear Dynamics, vol.~5,
World Scientific Publishing Co., Inc., River Edge, NJ, 2000.

\bibitem{Aw93}
M.~M.~Awartani,
\emph{An uncountable collection of mutually incomparable chainable continua},
Proc. Amer. Math. Soc. \textbf{118} (1993), no.~1, 239--245.

\bibitem{Ba97}
J.~Banks,
\emph{Regular periodic decompositions for topologically transitive maps},
Ergodic Theory Dynam. Systems \textbf{17} (1997), no.~3, 505--529.

\bibitem{Ba99}
J.~Banks,
\emph{Topological mapping properties defined by digraphs},
Discrete Contin. Dynam. Systems \textbf{5} (1999), no.~1, 83--92.

\bibitem{Bl86}
A.~M.~Blokh,
\emph{Dynamical systems on one-dimensional branched manifolds. I},
Teor. Funktsi\u{\i} Funktsional. Anal. i Prilozhen. (1986), no.~46, 8--18.

\bibitem{Bl87a}
A.~M.~Blokh,
\emph{Dynamical systems on one-dimensional branched manifolds. II},
Teor. Funktsi\u{\i} Funktsional. Anal. i Prilozhen. (1987), no.~47, 67--77.

\bibitem{Bl87b}
A.~M.~Blokh,
\emph{Dynamical systems on one-dimensional branched manifolds. III},
Teor. Funktsi\u{\i} Funktsional. Anal. i Prilozhen. (1987), no.~48, 32--46.

\bibitem{Bo71}
R.~Bowen,
\emph{Entropy for group endomorphisms and homogeneous spaces},
Trans. Amer. Math. Soc. \textbf{153} (1971), 401--414.

\bibitem{Bu97}
J.~Buzzi,
\emph{Specification on the interval},
Trans. Amer. Math. Soc. \textbf{349} (1997), no.~7, 2737--2754.

\bibitem{BFK17}
J.~Byszewski, F.~Falniowski, and D.~Kwietniak,
\emph{Transitive dendrite map with zero entropy},
Ergodic Theory Dynam. Systems \textbf{37} (2017), no.~7, 2077--2083.

\bibitem{DHMSS19}
M.~Dirb\'ak, R.~Hric, P.~Mali\v{c}k\'y, \v{L}.~Snoha, and V.~\v{S}pitalsk\'y,
\emph{Minimality for actions of abelian semigroups on compact spaces with a free interval},
Ergodic Theory Dynam. Systems \textbf{39} (2019), no.~11, 2968--2982.

\bibitem{DSS13}
M.~Dirb\'ak, \v{L}.~Snoha, and V.~\v{S}pitalsk\'y,
\emph{Minimality, transitivity, mixing and topological entropy on spaces with a free interval},
Ergodic Theory Dynam. Systems \textbf{33} (2013), no.~6, 1786--1812.

\bibitem{DY02}
T.~Downarowicz and X.~Ye,
\emph{When every point is either transitive or periodic},
Colloq. Math. \textbf{93} (2002), no.~1, 137--150.

\bibitem{Dr19}
T.~Drwi\c{e}ga,
\emph{Dynamika odwzorowa\'n w niskich wymiarach: Entropia, mieszanie i chaos},
Ph.D. thesis, Krak\'ow, 2019.

\bibitem{FKKL15}
F.~Falniowski, M.~Kulczycki, D.~Kwietniak, and J.~Li,
\emph{Two results on entropy, chaos and independence in symbolic dynamics},
Discrete Contin. Dyn. Syst. Ser.~B \textbf{20} (2015), no.~10, 3487--3505.

\bibitem{Fu67}
H.~Furstenberg,
\emph{Disjointness in ergodic theory, minimal sets, and a problem in Diophantine approximation},
Math. Systems Theory \textbf{1} (1967), 1--49.

\bibitem{HKO11}
G.~Hara\'nczyk, D.~Kwietniak, and P.~Oprocha,
\emph{A note on transitivity, sensitivity and chaos for graph maps},
J. Difference Equ. Appl. \textbf{17} (2011), no.~10, 1549--1553.

\bibitem{HKO14}
G.~Hara\'nczyk, D.~Kwietniak, and P.~Oprocha,
\emph{Topological structure and entropy of mixing graph maps},
Ergodic Theory Dynam. Systems \textbf{34} (2014), no.~5, 1587--1614.

\bibitem{HM14}
L.~Hoehn and C.~Mouron,
\emph{Hierarchies of chaotic maps on continua},
Ergodic Theory Dynam. Systems \textbf{34} (2014), no.~6, 1897--1913.

\bibitem{Ka88}
K.~Kawamura,
\emph{A direct proof that each Peano continuum with a free arc admits no expansive homeomorphisms},
Tsukuba J. Math. \textbf{12} (1988), no.~2, 521--524.

\bibitem{Ki58}
S.~Kinoshita,
\emph{On orbits of homeomorphisms},
Colloq. Math. \textbf{6} (1958), 49--53.

\bibitem{KKM21}
Z.~Ko\v{c}an, V.~Kurkov\'a, and M.~M\'alek,
\emph{Properties of dynamical systems on dendrites and graphs},
Internat. J. Bifur. Chaos Appl. Sci. Engrg. \textbf{31} (2021), no.~7, Paper No.~2150100, 10.

\bibitem{KS97}
S.~Kolyada and \v{L}.~Snoha,
\emph{Some aspects of topological transitivity---a survey},
Iteration theory (Opava, ECIT 94),
Grazer Math. Ber., vol.~334, Karl-Franzens-Univ. Graz, Graz, 1997, pp.~3--35.

\bibitem{KLO16}
D.~Kwietniak, M.~\L\k{a}cka, and P.~Oprocha,
\emph{A panorama of specification-like properties and their consequences},
Dynamics and numbers,
Contemp. Math., vol.~669, Amer. Math. Soc., Providence, RI, 2016, pp.~155--186.

\bibitem{KOT25}
D.~Kwietniak, P.~Oprocha, and J.~Tomaszewski,
\emph{On entropy of pure mixing maps on dendrites},
arXiv: 2504.05121 (2025).

\bibitem{MM14}
V.~Mart\'{\i}nez-de-la-Vega and P.~Minc,
\emph{Uncountable families of metric compactifications of the ray},
Topology Appl. \textbf{173} (2014), 28--31.

\bibitem{Mi23}
M.~Mihokova,
\emph{Minimal sets on continua with a dense free interval},
J. Math. Anal. Appl. \textbf{517} (2023), no.~1, Paper No.~126607, 17.

\bibitem{Ru17}
S.~Ruette,
\emph{Chaos on the interval},
University Lecture Series, vol.~67,
American Mathematical Society, Providence, RI, 2017.

\bibitem{Si92}
S.~Silverman,
\emph{On maps with dense orbits and the definition of chaos},
Rocky Mountain J. Math. \textbf{22} (1992), no.~1, 353--375.

\bibitem{SST22}
\v{L}.~Snoha, V.~\v{S}pitalsk\'y, and M.~Tak\'acs,
\emph{Generic chaos on dendrites},
Ergodic Theory Dynam. Systems \textbf{42} (2022), no.~6, 2108--2150.

\bibitem{Sp15}
V.~\v{S}pitalsk\'y,
\emph{Topological entropy of transitive dendrite maps},
Ergodic Theory Dynam. Systems \textbf{35} (2015), no.~4, 1289--1314.

\bibitem{We71}
B.~Weiss,
\emph{Topological transitivity and ergodic measures},
Math. Systems Theory \textbf{5} (1971), 71--75.

\end{thebibliography}

\begin{comment}

\end{comment}

\end{document}